\documentclass[11pt]{article}
\usepackage{amsmath}
\usepackage{amsthm}
\usepackage{amsfonts}
\usepackage{amssymb}
\usepackage{adjustbox}
\usepackage{authblk}
\usepackage{graphicx}
\usepackage{color}
\usepackage{rotating}
\usepackage{url}
\usepackage{algorithm}
\usepackage{algpseudocode}
\usepackage{epstopdf}
\usepackage[table,xcdraw]{xcolor}
\usepackage{adjustbox}
\usepackage[normalem]{ulem}
\useunder{\uline}{\ul}{}
\usepackage[T1]{fontenc}
\usepackage[utf8]{inputenc}
\usepackage[font=small,labelfont=bf]{caption}

\newtheorem{proposition}{Proposition}
\newtheorem{theorem}{Theorem}

\newtheorem{lemma}{Lemma}

\newtheorem{remark}{Remark}

\DeclareMathOperator*{\Argmax}{Arg\!\max}

\DeclareMathOperator*{\convex}{conv}

\DeclareMathOperator*{\rr}{\mathbb{R}}

\title{New SOCP relaxation and branching rule for bipartite bilinear programs}
\author[1]{Santanu S. Dey\thanks{santanu.dey@isye.gatech.edu}}
\author[1]{Asteroide Santana\thanks{asteroide.santana@gatech.edu}}
\author[2]{Yang Wang\thanks{yang.wang@ce.gatech.edu}}
\affil[1]{\small School of Industrial and Systems Engineering, Georgia Institute of Technology} 
\affil[2]{\small School of Civil and Environmental Engineering, Georgia Institute of Technology}

\begin{document}
\maketitle
\begin{abstract} A bipartite bilinear program (BBP) is a quadratically constrained quadratic optimization problem where the variables can be partitioned into two sets such that fixing the variables in any one of the sets results in a linear program. We propose a new second order cone representable (SOCP) relaxation for BBP, which we show is stronger than the standard SDP relaxation intersected with the boolean quadratic polytope. We then propose a new branching rule inspired by the construction of the SOCP relaxation. We describe a new application of BBP called as the finite element model updating problem, which is a fundamental problem in structural engineering. Our computational experiments on this problem class show that the new branching rule together with an polyhedral outer approximation of the SOCP relaxation outperforms a state-of-the-art commercial global solver in obtaining dual bounds. 
\end{abstract}

\section{Introduction: Bipartite bilinear program (BBP)}\label{sec:intro}
A quadratically constrained quadratic program (QCQP) is called as a bilinear optimization problem if every degree two term in the constraints and objective involves the product of two distinct variables.  For a given instance of bilinear optimization problem, one often associates a simple graph constructed as follows: The set of vertices corresponds to the variables in the instance and there is an edge between two vertices if there is a degree two term involving the corresponding variables in the instance formulation. Strength of various convex relaxations for bilinear optimization problems can be analyzed using combinatorial properties of this graph~\cite{LuedtkeNL12,BolandDKMR17,GupteKRW17}. 

When this graph is bipartite, we call the resulting bilinear problem as a bipartite bilinear program (BBP). In other words, BBP is an optimization problem of the following form:
\begin{eqnarray}\label{eq:BBP}
\begin{array}{rcl}
&\min & x^{\top}Q_0y + d_1^{\top} x + d_2^{\top} y\label{P}\\
&\textup{s.t.}& x^{\top}Q_ky + a_k^{\top}x + b_k^{\top}y + c_k = 0, \ k\in\{1, \dots, m\}\\
&& l\leq (x,y) \leq u \\
&&(x,y)\in {\rr}^{n_1+n_2},
\end{array}
\end{eqnarray}
where $n_1,n_2 \in \mathbb{Z}_+, \ Q_0, Q_k\in\rr^{n_1\times n_2}, \ d_1, a_k\in \rr^{n_1}, \ d_2, b_k\in\rr^{n_2}, \ c_k\in\rr$, ${\forall} k\in\{1, \dots, m\}$. The vectors $l,u\in\rr^{n_1+n_2}$ define the box constraints on the decision variables and, without loss of generality, we assume that $l_i=0, \ u_i = 1, \ {\forall} i\in \{1, \dots, n_1+n_2\}$. 
BBP~(\ref{eq:BBP}) may include bipartite bilinear inequality constraints, which can be converted into equality constraints by adding slack variables, and these slack variables will also be bounded since the original variables are bounded.

We note that BBP is a special case of the more general biconvex optimization problem~\cite{gorski2007biconvex}. BBP has many applications such as waste water management~\cite{faria2011novel,castro2015tightening,galan1998optimal}, pooling problem~\cite{GupteADC17,haverly1978studies}, and supply chain~\cite{nahapetyan2008bilinear}. 

\section{Our results}
\subsection{Second order cone representable relaxation of BBP}\label{sec:SOCPintro}

A common and successful approach in integer linear programing is to generate cutting-planes implied by single constraint relaxation, see for example~\cite{crowder1983solving,marchand2001aggregation,dey2017analysis,Bodur2017}. We take a similar approach here. We begin by examining one row relaxation of BBP, that is, we study the convex hull of the set defined by a single constraint defining the feasible region of (\ref{eq:BBP}). Our first result is to show that the convex hull of this set is second order cone (SOCP) representable in the extended space, where we have introduced new variables $w_{ij}$ for $x_iy_j$. We formally present this result next.

\begin{theorem}\label{thm:conv}
Let $n_1,n_2 \in \mathbb{Z}_{+}$, $V_1 \in \{1,\dots, n_1\}$, $V_2 \in \{1, \dots, n_2\}$, and $E \subseteq V_1 \times V_2$. Consider the one-constraint BBP set $$S: = \left\{ (x,y,w)\in[0,1]^{n_1 + n_2 + |E|} \, \left| \, \begin{array}{l}\sum_{(i,j)\in E}q_{ij}w_{ij} + \sum_{i\in V_1}a_i x_i + \sum_{j\in V_2}b_j y_j + c = 0, \\ w_{ij}=x_iy_j, \ {\forall} (i,j)\in E \end{array}\right\}\right. .$$ Then:
\begin{enumerate}
\item[(i)] Let $(\bar{x},\bar{y}, \bar{w})$ be an extreme point of $S$. Then, there exists $U \subseteq V_1 \cup V_2$, of the form
\begin{enumerate}
\item $U = \{i_0,j_0\}$ where $(i_0,j_0)\in E$, or  
\item $U = \{i_0\}$ where $i_0 \in V_1$ is an isolated node, or 
\item $U = \{j_0\}$ where $j_0 \in V_2$ is an isolated node,
\end{enumerate}
such that $\bar{x}_{i}\in\{0,1\}, \ \forall i\in V_1 \setminus U$, and $\bar{y}_{j}\in\{0,1\}, \ \forall j \in V_2\setminus U$.
\item[(ii)] $\textup{conv}(S)$ is SOCP-representable.
\end{enumerate}
\end{theorem}

A proof of Theorem~\ref{thm:conv} is presented in Section \ref{sec: proof of thm 1}.
\begin{remark}\label{rem:size}
In Theorem~\ref{thm:conv}, part (ii) follows from part (i). For any given choice of $U$, we first fix all the variables to $0$ or $1$ except for those in $U$. It is then shown that the convex hull of the resulting set is SOCP-representable and we obtain (ii) by convexifying the union of a finite set of SOCP representable sets. 

It is easy to see that the number of distinct $U$ sets is $\mathcal{O}(n_1n_2)$, and the number of possible fixings is $\mathcal{O}(2^{n_1 + n_2})$. Thus, the number of resulting SOCP representable objects is $\mathcal{O}(n_1n_22^{n_1 + n_2})$.
\end{remark}

We note that the literature in global optimization theory has many results on convexifying functions, see for example~\cite{al1983jointly,Rikun1997,meyer2005convex,tawarmalani2002convexification,tuy2016convex}. However, as is well-known, replacing a constraint $f(x) = b$ by $\{x\,|\, \hat{f}(x) \geq b, \ \breve{f}(x) \leq b\}$ where $\hat{f}$ and $\breve{f}$ are the concave and convex envelop of $f$, does not necessarily yield the convex hull of the set $\{x \,|\, f(x) = b\}$. There are relatively lesser number of results on convexification of sets~\cite{tawarmalani2013explicit,nguyen2013deriving,nguyen2011convexification,tawarmalani2010strong}. 
Theorem~\ref{thm:conv} generalizes results presented in~\cite{Tawarmalani2010,akshayguptethesis,kocuk2017matrix} and is related to results presented in~\cite{davarnia2017simultaneous}.

\paragraph{The SOCP relaxation for the feasible region of the general BBP (\ref{eq:BBP})} that we propose, henceforth referred as $S^{SOCP}$, is the intersection of the convex hull of each of the constraints of (\ref{eq:BBP}). Formally:
$$S^{SOCP} = \bigcap_{k = 1}^m \textup{conv}(S_k),$$ where $S_k = \{(x,y,w)\in [0,1]^{n_1 \times n_2 \times |E|}\, | \, x^{\top}Q_ky + a_k^{\top}x + b_k^{\top}y + c_k = 0, w_{ij} = x_iy_j \ \forall (i,j) \in E \}$ and $E$ is the edge set of the graph corresponding to the BBP instance (and not just of one row). As an aside, note that $S^{SOCP}$ can be further strengthened by adding the convex hull of single row BBP sets arrived by taking linear combinations of rows. 

Next we discuss the strength of $S^{SOCP}$ vis-\'a-vis the strength of other standard relaxations. Consider the following two standard relaxations of the feasible region of BBP (\ref{eq:BBP}): Let $S^{SDP}$ be the standard semi-definite programming (SDP) relaxation and let
\begin{eqnarray}
S^{QBP} &:=& \{(x, y, w)\in[0,1]^{n_1+n_2+|E|} \,|\, \sum_{(ij) \in E} (Q_k)_{ij}w_{ij} + a_k^{\top}x + b_k^{\top}y + c_k = 0  \ k \in \{1, \dots, m\}\} \nonumber\\
&& \bigcap \textup{ conv}\left(\{(x, y, w) \in[0,1]^{n_1+n_2+ |E|}\,|\, w_{ij} = x_i y_j \ \forall (i,j) \in E\}\right).\label{eq:QBP0}
\end{eqnarray}
Note that $S^{QBP}$ is a polyhedral set, since the second set in the right-hand-side of (\ref{eq:QBP0}) is equal to the Boolean Quadratic Polytope~\cite{Burer2009}. Two well-known classes of valid inequalities for this set are the McCormick's inequalities~\cite{al1983jointly} and the triangle inequalities {\cite{padberg1989boolean}}.
\begin{theorem}\label{thm:str}
For any BBP, we have that 
$$\textup{proj}_{x, y, w} \left(S^{SDP}\right) \bigcap S^{QBP} \supseteq S^{SOCP}.$$
\end{theorem}
A proof of Theorem~\ref{thm:str} is presented in Section \ref{sec:thm2}.
\begin{remark}\label{remark: sparse graph}
It is possible to show that the convex hull of one row BBP is SOCP representable, even without introducing the $w$ variables. Thus, it is possible to construct, similar to $S^{SOCP}$, a SOCP-representable relaxation of BBP, without introducing $w$ variables. However, this SOCP relaxation would be weaker. In particular, we are unable to prove the corresponding version of Theorem~\ref{thm:str} for this SOCP relaxation. The strength of $S^{SOCP}$ relaxation is due to the fact that the extended space $w$ variables `interact' from different constraints.
\end{remark}
We note that other SOCP relaxations for QCQPs have been proposed~\cite{kim2001second,Burer2014}. However, these are all weaker than the standard SDP relaxation. 

We also note that it is polynomial time to optimize on $S^{SDP}$, although the tractability of solving SDPs in practice is still limited. On the other hand, solvers for SOCPs are significantly better in practice. It is NP-hard to optimize on $S^{QBP}$, although as discussed in Remark~\ref{rem:size}, the size of the extended formulation to obtain $S^{SOCP}$ is exponential in size.

\subsection{A new branching rule}
For details about general branch-and-bound scheme for global optimization see, for example, \cite{Ryoo1996}. Inspired by the convex relaxation described in Section \ref{sec:SOCPintro}, we propose a new rule for partitioning the domain of a given variable in order to produce two branches. Details of this new proposed branching rule together with node selection and variable selection rules that we used in our computational experiments are presented in Section~\ref{sec:bb}. 

Here, we sketch the main ideas behind our new proposed branching rule.
Suppose we have decided to branch on the variable $x_1$. As explained in Remark~\ref{rem:size}, the convex hull of the one constraint set is obtained by taking the convex hull of union of sets obtained by fixing all but two (or one) variables. If we are branching on $x_1$, we examine all such two-variable sets involving $x_1$ obtained from each of the constraints. For each of these sets, there is an ideal point to divide the range of $x_1$ so that the sum of the volume of the two convex hulls of the two-dimensional sets corresponding to the two resulting branches is minimized. (See recent papers on importance of volume minimization in branch-and-bound algorithm~\cite{speakman2017branching}). We present a heuristic to find an  ``ideal range". We collect all such ideal ranges corresponding to all the two-dimensional sets involving $x_1$. Then we present a heuristic to select one points (based on corresponding volume reduction) to finally partition the domain of $x_1$.  We also use similar arguments to propose a new variable selection rule. 

\subsection{A new application of BBP and computational experiments}

A new application of BBP, which motivated our work presented here, is called as the \textit{finite element model updating problem}, which is a fundamental methodological problem in structural engineering. See Section~\ref{sec: finite element model} for a description of the problem. All the new methods we develop here are tested on instances of this problem.

Due to the large size of $S^{SOCP}$, in practice, we consider a lighter version of this relaxation. In particular, we write the extended formulation of each row of BBP corresponding only to the variables in that row (see details in Section \ref{sec: lighter version}). As our instances are row sparse, the resulting SOCP relaxation can be solved in reasonable time. Unfortunately, there are no theoretical guarantees for the bounds of this light version of the relaxation. After some preliminary experimentation, we observed that a polyhedral outer approximation of the SOCP relaxation produces similar bounds but solves much faster. Therefore, we used this linear programming (LP) relaxation in our experiments. Details of this outer approximation is presented in Section~\ref{sec:Polyhedralrelaxation}.

Our computational experiments are aimed at making three comparisons. First, we examined the quality of the dual bound produced at root node via our new method (polyhedral outer approximation of SOCP relaxation) against SDP, McCormick, and SDP together with McCormick inequalities. The bounds produced are better for the new method. Second, we test the performance of the new branching rule against traditional branching rules. Our experiments show that the new branching rule significantly out performs the other branching rules. Finally, we compare the performance of our naive branch-and-bound implementation against BARON. In all instances, we close significantly more gap in equal amount of time. All these results are discussed in detail in Section~\ref{sec:compres}.

\section{Second order cone representable relaxation and its strength}\label{sec:thm1}
\subsection{Proof of Theorem~\ref{thm:conv}} \label{sec: proof of thm 1}
Consider the bipartite graph $G=(V_1,V_2,E)$ defined by the set of vertices $V_1 = \{1, \dots, n_1\}$ and $V_2= \{1, \dots, n_2\}$ which is associated to the equation
\begin{align}
\sum_{(i,j)\in E}q_{ij}x_i y_j  + \sum_{i\in V_1}a_i x_i + \sum_{j\in V_2}b_j y_j+ c = 0.\tag{EQ} \label{eq_single}
\end{align}
In this section, we prove that the convex hull of the set
\begin{align}
S = \{(x,y, w)\in[0,1]^{{n_1}+{n_2} + |E|} \,|\, (\ref{eq_single}), \ w_{ij} = x_iy_j \ \forall (i,j) \in E \}. \label{set: S EQ}
\end{align}
is SOCP representable. In addition, the proof provides {an} implementable procedure to obtain $\convex(S)$. The key idea underlying this result is the fact that, at each extreme point of $S$, at most two variables are not fixed to 0 or 1 and, once all variables but two (or one) are fixed, the convex hull of the resulting object is SOCP representable in $\rr^2$ (or $\rr$). Hence, $\convex(S)$ can be written as the convex hull of an union of SOCP representable sets.

\subsubsection{Preliminary results}
First we present a few preliminary results that will be used to prove that $\convex(S)$ is SOCP representable.

\begin{lemma}\label{lemma: convex of convex eq}\cite{Tawarmalani2013}
Let $f:[0,1]^n\to\mathbb{R}$ be a continuous function and $B \subseteq [0, 1]^n$ be a convex set. Then
$$
\convex({\{x\in B\,|\, f(x)=0\}})= \convex \left({\{x\in B\,|\, f(x)\leq 0\}}\right) \bigcap \convex({\{x\in B \,|\, f(x)\geq 0\}}).
$$
\end{lemma}

\begin{lemma}\label{lemma: reverse convex set}\cite{Hillestad1980}
Let $f:[0,1]^n\to\mathbb{R}$ be a convex function. Then 
$$
G:= \convex({\{x\in [0,1]^n\,|\,  f(x)\geq 0\}}),
$$
is a polytope. Indeed, $G$ can be obtained as the convex hull of finite number of points obtained as follows: fix all but one variable to $0$ or $1$ and solve for $f(x) = 0$. 
\end{lemma}

\begin{lemma}\label{lemma: convex of union of convex sets}\cite{ben2001lectures}
Let $T\subset \rr^n$ be a compact set and $\{T_k\}_{k\in K}$ be a partition of the set of all extreme points of $T$. Then,
\begin{align}
\convex(T) = \convex\left(\bigcup_{k\in K} T_k \right) = \convex \left(\bigcup_{k\in K}\convex(T_k) \right).
\end{align}
In addition, if $\convex(T_k)$ is a SOCP representable set for every $k\in K$, then $\convex(T)$ is also a SOCP representable set.
\end{lemma}

\begin{lemma}\label{lemma: convex of intersection with affine set}
Let $B=\{(x,w)\in[0,1]^{n} \times \mathbb{R}\,|\,  x\in B_0, \ w=l^{\top}x+l_0\}$, where $B_0\subseteq\rr^{n}$, and $l^{\top}x + l_0$ is an affine function of $x$. Then,
$$\convex(B) = \{(x,w)\in[0,1]^{n} \times \mathbb{R}\,|\, x\in \convex(B_0), \ w=l^{\top}x+l_0\}.$$
\end{lemma}
\begin{proof} We assume $B_0$ is non-empty, otherwise, there is nothing to prove. 
Let $(x,w)\in \convex(B)$. Then there exist $(x^i,w^i)\in B$ and $\lambda_i\geq 0, \ \forall i\in \{1, \dots, n+2\}$, such that $\sum_{i=1}^{n+2}\lambda_i=1$, $x = \sum_{i=1}^{n+2}\lambda_i x^i$ and $w = \sum_{i=1}^{n+2}\lambda_i w^i$. It follows by the definition of $B$ that $x^i\in B_0, \ \forall i\in \{1, \dots, n+2\}$, and hence $x\in \convex(B_0)$. It also follows from the definition of $B$ that $w^i = l^{\top}x^i+l_0, \ \forall i\in \{1, \dots, n+2\}$, and hence
$$w = \sum_{i=1}^{n + 2} \lambda_iw^i = \sum_{i=1}^{n + 2}\lambda_i(l^{\top}x^i+l_0) = l^{\top}\left(\sum_{i=1}^{n+ 2}\lambda_ix^i\right)+l_0 = l^{\top}x + l_0.$$
Conversely, let $(x,w)$ be such that $x\in \convex(B_0)$ and $w=l^{\top}x+l_0$. Then, there exist $x^i\in B_0$ and $\lambda_i\geq 0, \ \forall i\in \{1, \dots, n+1\}$, such that $\sum_{i=1}^{n + 1}\lambda_i=1$, $x = \sum_{i=1}^{n + 1}\lambda_i x^i$. Define $w^i = l^{\top}x^i+l_0, \ \forall i\in \{1, \dots, n+1\}$. Then $(x^i,w^i)\in B, \ \forall i\in \{1, \dots, n+1\}$. In addition, 
$$
w = l^{\top}x+l_0 = l^{\top}\left(\sum_{i=1}^{n + 1}\lambda_i x^i\right) + l_0 = \sum_{i=1}^{n + 1}\lambda_i (l^{\top}x^i + l_0) = {\sum_{i=1}^{n+1}}\lambda_i w^i,
$$
which completes the proof.
\end{proof}

\subsubsection{Proof of part (i) of Theorem~\ref{thm:conv}}
We restate part (i) of Theorem~\ref{thm:conv} next for easy reference:
\begin{proposition}\label{prop: extreme points of a single eq}
Let $(\bar{x},\bar{y}, \bar{w})$ be an extreme point of the set $S$ defined in (\ref{set: S EQ}). Then, there exists $U \subseteq V_1 \cup V_2$, of the form 
\begin{enumerate}
\item $U = \{i_0,j_0\}$ where $(i_0,j_0)\in E$, or, 
\item $U = \{i_0\}$ where $i_0 \in V_1$ is an isolated node, or,
\item $U = \{j_0\}$ where $j_0 \in V_2$ is an isolated node,
\end{enumerate}
such that $\bar{x}_{i} \in\{0,1\}, \ \forall i\in V_1 \setminus U$, and $\bar{y}_{j}, \ \forall j \in V_2\setminus U$.
\end{proposition}
\begin{proof}
To prove by contradiction, suppose without loss of generality that $0 < \bar{x}_{1}, \bar{x}_{2} < 1$. Consider the system of equations
\begin{eqnarray}
\bar{a}_1x_1 + \bar{a}_2x_2 + \bar{c} & = & 0, \nonumber \\
w_{1j} - x_1 \bar{y}_j & = & 0  \ \forall j: (1,j) \in E \nonumber\\
w_{2j} - x_2 \bar{y}_j & = & 0 \ \forall j: (2,j) \in E, \nonumber
\end{eqnarray}
obtained by fixing $x_i=\bar{x}_i, \ y_j=\bar{y}_j$ in (\ref{set: S EQ}), $w_{ij} = \bar{x}_i\bar{y}_j$ $\forall i\in V_1\setminus\{1,2\}, \ \forall j\in V_2$. Since $(\bar{x}_1,\bar{x}_2)$ is in the relative interior of $\{(x_1,x_2)\in [0,1]^2 \,|\, \bar{a}_1x_1 + \bar{a}_2x_2 + \bar{c} = 0\}$,  $(\bar{x},\bar{y}, \bar{w})$ cannot be an extreme point of $S$.
\end{proof}

\subsubsection{Proof of part (ii) of Theorem~\ref{thm:conv}}
First, we prove that the two-variable sets we encounter after fixing variables are SOCP representable. 

\begin{proposition}\label{prop: hyp is SOCP} Let 
$
S_0 = \{(x,y)\in[0,1]^2\,|\, \ ax + by + q x y + c = 0\}.
$ 
Then, $\convex(S_0)$ {is SOCP representable}.
\end{proposition}
\begin{proof}
We may assume $S_0\neq \emptyset$ and $q\neq 0$, otherwise the result follows trivially. 
Define $r=-b/q, \ s=-a/q$ and $\tau=(ab-cq)/q^2$ to write  $ax + by + q x y + c =0$ equivalently as 
\begin{align}
(x-r)(y-s)=\tau.\label{hyp_eq}
\end{align}
If $\tau=0$, then (\ref{hyp_eq}) is equivalent to $x=r$ or $y=s$. In this case, $S_0 = \{(x,y)\in[0,1]^2\,|\, x=r\}\cup \{(x,y)\in[0,1]^2\, |\, y=s\}$ and hence $\convex(S_0)$ is a polytope. 
Suppose $\tau > 0$ (if $\tau < 0$, we multiply (\ref{hyp_eq}) by $-1$ and repeat the same proof with $x-r$ and $\tau$ replaced with $-(x-r)$ and $-\tau$). Either $x-r,y-s\geq 0$ or $x-r,y-s \leq 0$. Thus, $S_0 = S_0^>\cup S_0^<$, where $S_0^>=\{(x,y)\in[0,1]^2\,|\, x-r,y-s \geq 0, \ (\ref{hyp_eq})\}$ and $S_0^<=\{(x,y)\in[0,1]^2 \,|\, x-r,y-s \leq 0, \ (\ref{hyp_eq})\}$. Next, we show that if $S_0^>\neq \emptyset$, then $\convex(S_0^>)$ is SOCP representable.
Using that $4uv = (u+v)^2-(u-v)^2$, we can rewrite (\ref{hyp_eq}) as
\begin{align*}
\sqrt{[(x-r)-(y-s)]^2 + (2\sqrt{\tau})^2} = (x-r)+(y-s).
\end{align*}
It now follows from Lemma~\ref{lemma: convex of convex eq} that $\convex(S_0^>) = \convex(S_{1}^>) \cap \convex(S_{2}^>)$, where 
\begin{align*}
S_1^> = \{(x,y)\in [0,1]^2 \,|\, x-r,y-s \geq 0, \ \sqrt{[(x-r)-(y-s)]^2 + (2\sqrt{\tau})^2} \leq (x-r)+(y-s)\} \ \\
S_2^> = \{(x,y)\in [0,1]^2 \,|\, x-r,y-s\geq 0, \ \sqrt{[(x-r)-(y-s)]^2 + (2\sqrt{\tau})^2} \geq (x-r)+(y-s)\}.
\end{align*} 
Notice that $S_1^>$ is SOCP representable. Also, as the square root term in the definition of $S_2^>$ is a convex function in $x$ and $y$, it follows from Lemma~\ref{lemma: reverse convex set} that $S_2^>$ is a polytope. Thus, $\convex(S_0^>)$ is SOCP representable. Similarly, we can prove that $\convex(S_0^<)$ is SOCP by repeating the arguments above after replacing $x-r,y-s$ with $-(x-r),-(y-s)$. Therefore, $\convex(S_0)=\convex(S_0^>\cup S_0^<)=\convex(\convex(S_0^>)\cup \convex(S_0^<))$ {is SOCP representable} by Lemma~\ref{lemma: convex of union of convex sets}.
\end{proof}

\begin{proposition}\label{prop: parabola is SOCP}
Let 
$
S_0 = \{(x,y)\in[0,1]^2 \,|\, \ y = a_0 + a_1 x + a_2 x^2\}.
$ 
Then $\convex(S_0)$ {is SOCP representable}.
\end{proposition}
\begin{proof}
We may assume $S_0\neq \emptyset$ and $a_2\neq 0$, otherwise the result follows trivially. By completing squares, we can write $y = a_0+a_1 x + a_2 x^2$ equivalently as 
$
(x + 0.5a_1/a_2)^2-(a_1/2a_2)^2  + a_0/a_2 = y/a_2,
$
and then as
\begin{align}
(x+\bar{a})^2 = t \ \Leftrightarrow \ \sqrt{(x+\bar{a})^2+\left(\frac{t-1}{2	}\right)^2} = \frac{t+1}{2},
\end{align}
where $\bar{a}=0.5a_1/a_2, \ t= y/a_2 + (a_1/2a_2)^2 - a_0/a_2$, using that $4t = (t+1)^2 - (t-1)^2$.
It now follows from Lemma~\ref{lemma: convex of convex eq} that $\convex(S_0) = \convex(S_{1}) \cap \convex(S_{2})$, where 
\begin{align*}
S_1 = \{(x,y)\in [0,1]^2\,|\, \sqrt{(x+\bar{a})^2+\left(\frac{t-1}{2	}\right)^2} \leq \frac{t+1}{2}\} \ \\
S_2 = \{(x,y)\in [0,1]^2 \,|\,\sqrt{(x+\bar{a})^2+\left(\frac{t-1}{2	}\right)^2} \geq \frac{t+1}{2}\}.
\end{align*} 
Notice that $S_1$ is SOCP representable. Also, as the square root term in the definition of $S_2$ is a convex function in $x$ and $y$ (because $t$ is an affine function of $y$), it follows from Lemma~\ref{lemma: reverse convex set} that $S_2$ is a polytope. Thus, $\convex(S_0)$ is SOCP representable.
\end{proof}

\begin{proposition}\label{prop: hyp is SOCP in space of w} Let 
$
S_0 = \{(x,y,w)\in[0,1]^3\,|\, ax + by + q w + c = 0, \ w = x y\}.
$ 
Then, $\convex(S_0)$ {is SOCP representable}.
\end{proposition}
\begin{proof}
If $q\neq 0$, then we can write
\begin{eqnarray}\label{eq:soqneq0}
S_0 = \{(x,y,w)\in[0,1]^2 \times \mathbb{R}\,|\, (x,y)\in B_0, \ w = (-c-ax-by)/q\},
\end{eqnarray}
where $B_0 = \{(x,y)\in[0,1]^2 \,|\,  ax + by + q x y + c = 0\}$. (Note that the bounds on $w$ are automatically enforced in (\ref{eq:soqneq0}) and it is sufficient to say $w \in \mathbb{R}$). 
Hence, by Proposition~\ref{prop: hyp is SOCP} and Lemma~\ref{lemma: convex of intersection with affine set}, $\convex(S_0)$ {is SOCP representable}. 

Now, suppose $q=0$. Four cases:
(i) $a,b=0$. In this case, we may assume $c=0$, otherwise $S_0=\emptyset$. Then, 
$
S_0 = \{(x,y,w)\in[0,1]^3 \,|\, w = xy\},
$
in which  case $\convex(S_0)$ is a well known polytope given by the McCormick envelope.
(ii) $a = 0$, $b \neq 0$. In this case, if $-c/b \notin [0,1]$, then $S_0$ is infeasible. Otherwise, this case is trivial. 
(iii) $a \neq 0$, $b = 0$. Similar to previous case. 
(iv) $a\neq 0$ and $b\neq 0$. In this case, we can solve $ax + by + c = 0$ for $x$, i.e. $x = (-c - by)/a$. Let $[\alpha, \beta]$ be the bounds on $y$ such that the line $ax + by + c = 0$ intersects the $[0,1]^2$ box. If $\alpha = \beta$, then we can set $y = \alpha$ and the result follows trivially. Otherwise, substitute in $w=xy$ to rewrite $S_0$ as following
$$
S_0 = \{(x,y,w)\in\mathbb{R} \times [\alpha, \beta] \times [0,1] \,|\, (y,w)\in B_0, \ x = (-by-c)/a\},
$$
where 
$B_0 = \{(y,w)\in[\alpha, \beta] \times [0,1]\,|\, w = (-c/a)y-(b/a)y^2\}$.
Now, it is straightforward via Proposition~\ref{prop: parabola is SOCP} (affinely scale $y$ to have bound of $[0, 1]$) and Lemma~\ref{lemma: convex of intersection with affine set} that $\convex(S_0)$ is a SOCP representable set.
\end{proof}

Now we are ready to prove part (ii) of Theorem~\ref{thm:conv}. 

\begin{proposition}\label{thm: convex of single eq is SOCP in the space of w}
Let ${S}$ be the set defined in (\ref{set: S EQ}). Then $\convex({S})$ is SOCP representable.
\end{proposition}
\begin{proof}
By Proposition~\ref{prop: extreme points of a single eq}, we can fix various sets of $x$ and $y$ variables that corresponds to the $U$ sets and prove that the convex hull of each of these sets is SOCP representable. Case (i): $|U| = 1$. In this case, the set of unfixed variables satisfy a set of linear equations. Thus this set is clearly SOCP representable. Case (ii): $U = \{(i_0, j_0)\}$, where $(i_0, j_0)\in E$. In this case, the set of unfixed variables satisfy the following constraints:
\begin{eqnarray}
ax_{i_0} + by_{j_0} + q w_{i_0 j_0} + c & = & 0, \\
w_{i_0 j_0} &=& x_{i_0} y_{j_0}\\
w_{i j_0} &=& \bar{x}_i y_{j_0} \ \forall (i,j_0) \in E, i \neq i_0\\
w_{i_0 j} &=& \bar{y}_j x_{i_0} \ \forall (i_0,j) \in E, j \neq j_0,
\end{eqnarray}
where the bound constraints on $w_{i j_0}$ and $w_{i_0 j}$ variables are not needed explictly. Thus, by Proposition~\ref{prop: hyp is SOCP in space of w} and Lemma~\ref{lemma: convex of intersection with affine set}, the above set is SOCP representable. Thus, by Lemma ~\ref{lemma: convex of union of convex sets}, we obtain that $\convex({S})$ is SOCP representable.
\end{proof}

\subsection{Proof of Theorem~\ref{thm:str}}\label{sec:thm2}

In order to prove Theorem~\ref{thm:str} it is sufficient to prove that:
\begin{eqnarray}\label{eq:SDP}
\textup{proj}_{x, y, w}\left(S^{SDP}\right) \supseteq S^{SOCP}
\end{eqnarray} and 
\begin{eqnarray}\label{eq:QBP}
S^{QBP} \supseteq S^{SOCP}.
\end{eqnarray}
We prove these two containments next. 
\begin{proposition}
For any BBP, (\ref{eq:SDP}) holds. 
\end{proposition}
\begin{proof}
In order to prove (\ref{eq:SDP}), it is convenient to introduce some notation. Let $H$ be the matrix variable representing $\left[\begin{array}{c}x \\ y\end{array}\right][x^{\top} y^{\top}]$. We write $w = \textup{proj}_E(H)$, to imply that if $(i,j) \in E$, then $w_{ij} = \frac{1}{2}\left(H_{i(j+n_1)} + H_{(j + n_1)i}\right)$. 

Then the standard SDP relaxation may be written as:
\begin{eqnarray} 
\sum_{ij \in E}(Q_k)_{ij}w_{ij} + a_k^{\top}x + b_k^{\top}y + c_k & = & 0, \ k\in\{1, \dots, m\} \label{eq:row}\\
\textup{proj}_E(H) & = & w \label{eq:projw}\\
\left[\begin{array}{cc} H & [x^{\top} y^{\top}] \\ \left[\begin{array}{c}x \\ y\end{array}\right] & 1\end{array} \right] &\succeq& 0. \label{eq:sdpcon}
\end{eqnarray} 
Let $$T^k: = \{ (x, y, H, w)\,|\, (\ref{eq:row}) \textup{ corresponding to }k, (\ref{eq:projw}), \textup{ and }(\ref{eq:sdpcon})\}$$ and as before let $$S^k:= \{(x, y, w)\,|\, (\ref{eq:row}) \textup{ corresponding to }k, w_{ij} = x_i y_j \ \forall (ij) \in E\}.$$ Then by construction 
\begin{eqnarray}
\textup{proj}_{x,y,w}\left(T^k\right) \supseteq \textup{conv}(S^k)\label{eq:perrow}.
\end{eqnarray} 
Next we need the following:
\paragraph{Claim 1} $\bigcap_{k =1}^m \textup{proj}_{x,y,w}\left(T^k\right) =  \textup{proj}_{x,y,w}\left(\bigcap_{k =1}^m \left(T^k\right)\right)$: Trivially we have that, $$\bigcap_{k =1}^m \textup{proj}_{x,y,w}\left(T^k\right) \supseteq \textup{proj}_{x,y,w}\left(\bigcap_{k =1}^m \left(T^k\right)\right),$$ holds. 

We now verify the converse. 
For some $(\bar{x},\bar{y}, \bar{w}) \in \textup{proj}_{x,y,w}\left(T^{k}\right)$, let $$\mathcal{H}^k(\bar{x}, \bar{y}, \bar{w}):= \left\{H \,|\, (\bar{x}, \bar{y}, \bar{w}, H) \in T^k\right\}.$$ Then observe that  $\mathcal{H}^k(\bar{x}, \bar{y}, \bar{w})$ is the set of matrices $H$ satisfying
\begin{eqnarray}
\textup{proj}_E(H) &=& \bar{w}\\
\left[\begin{array}{cc} H & [\bar{x}^{\top} \bar{y}^{\top}] \\ \left[\begin{array}{c}\bar{x} \\ \bar{y}\end{array}\right] & 1\end{array} \right] &\succeq& 0.
\end{eqnarray}
Thus $\mathcal{H}^k(\bar{x},\bar{y}, \bar{w})$ is independent of $k$, i.e. if $(\bar{x},\bar{y}, \bar{w}) \in \bigcap_{k = 1}^m\textup{proj}_{x,y,w}\left(T^{k}\right)$  then $\mathcal{H}^{k_1}(\bar{x},\bar{y},\bar{w}) = \mathcal{H}^{k_2}(\bar{x},\bar{y},\bar{w})$ for all $k_1 \neq k_2$. Therefore in particular, if $(\bar{x},\bar{y}, \bar{w}) \in \bigcap_{k = 1}^m\textup{proj}_{x,y,w}\left(T^{k}\right) $, then there exists $\bar{H}$ such that $(\bar{x}, \bar{y}, \bar{w}, \bar{H}) \in \bigcap_{k = 1}^mT^{k}$. Thus,  $(\bar{x},\bar{y}, \bar{w}) \in \textup{proj}_{x,y,w}\left( \bigcap_{k = 1}^mT^{k}\right).$~$\diamond$

Now, we return to the proof of the original statement. 
Intersecting (\ref{eq:perrow}) for all $k \in \{1, \dots, m\}$ we obtain, 
\begin{eqnarray}
\textup{proj}_{x,y,w}\left(S^{SDP}\right) = \textup{proj}_{x,y,w}\left(\bigcap_{k =1}^m \left(T^k\right)\right)  = \bigcap_{k =1}^m \textup{proj}_{x,y,w}\left(T^k\right) \supseteq  \bigcap_{k =1}^m \textup{conv}(S^k) = S^{SOCP}, \nonumber
\end{eqnarray}
where the first equality is by definition of $S^{SDP}$, the second equality via Claim 1, the inequality is due to (\ref{eq:perrow}) and the last equality is by definition of $S^{SOCP}$.
\end{proof}

\begin{proposition}
For any BBP, (\ref{eq:QBP}) holds. 
\end{proposition}
\begin{proof}
Recall that $S^{QBP}$ is the set
\begin{eqnarray}
&& \left\{(x, y, w)\in[0,1]^{n_1+n_2+|E|} \,|\, \sum_{(ij) \in E} (Q_k)_{ij}w_{ij} + a_k^{\top}x + b_k^{\top}y + c_k = 0  \ k \in \{1, \dots, m\}\right\} \label{eq:row1}\\
&&\bigcap \textup{ conv}\left(\{(x, y, w) \in[0,1]^{n_1+n_2+ |E|}\,|\, w_{ij} = x_i y_j \ \forall (i,j) \in E\}\right).\label{eq:QBP1}
\end{eqnarray}
Let $$T^k:= \{(x, y, w)\in[0,1]^{n_1+n_2+|E|} \,|\, (\ref{eq:row1}) \textup{ corresponding to }k, (\ref{eq:QBP1})\}$$ and let $$S^k:= \{(x, y, w)\,|\, (\ref{eq:row}) \textup{ corresponding to }k, w_{ij} = x_i y_j \ \forall (ij) \in E\}.$$ Then by construction 
\begin{eqnarray}
T^k \supseteq \textup{conv}(S^k)\label{eq:perrow1}.
\end{eqnarray} 
Intersecting (\ref{eq:perrow1}) for all $k \in \{1, \dots, m\}$ we obtain, 
$$S^{QBP} = \bigcap_{k = 1}^m T^k  \supseteq \bigcap_{k = 1}^m\textup{conv}(S^k) = S^{SOCP}.$$
\end{proof}

\section{Proposed branch-and-bound algorithm}\label{sec:bb}
In this section, we discuss some details of our proposed branch-and-bound algorithm to solve BBP~(\ref{eq:BBP}).

\subsection{Node selection and partitioning strategies}

The most common node selection rule used in the literature is the so-called \textit{best-bound-first}, in which a node with the least lower bound (assuming minimization) is chosen for branching. Other rules may include selection of nodes that have the potential of identifying good feasible solutions earlier. In our computational experiments, we only use best-bound-first rule. Also, we use the most simple partitioning operation: rectangular. Example of other operation adopted in the literature are conical and simplicial~\cite{Linderoth2005}.

\subsection{Variable selection and point of partitioning}\label{sec: variable and branching point selection}

A simple rule for variable selection is to choose a variable with largest range. Another common rule is to prioritize the variable that is most responsive for the approximation error of nonlinear terms. For example, suppose we are optimizing in the extended space of $(x,y,w)$, then we could chose $x_i$ (or $y_j$) for which the absolute error $|\bar{w}_{ij}-\bar{x}_i\bar{y}_j|$ is maximized over the set of all possible pairs $(i,j)$, where $(\bar{x},\bar{y},\bar{w})$ is the relaxation solution for the current node. We refer to this rule as the \textit{gap-error-rule}.

Once the variable is selected, say $x_1$ (without loss of generality), we can list three standard rules for choosing the partitioning point:\\
\textit{Bisection}: partition at the mid point of the domain of $x_1$ in the current node.\\
\textit{Maximum-deviation}: partition at $\bar{x}_1$, where $(\bar{x},\bar{y},\bar{w})$ is the relaxation solution for the current node.\\
\textit{Incumbent}: partition at $x^*_1$, where $(x^*,y^*,w^*)$ is the current best feasible solution, if $x^*_1$ is in the range of $x_1$ in the current node.

Combination of the above rules have also been proposed. For example, Tawarmalani et al. \cite{Sahinidis2005} propose a rule that is a convex combination of bisection and maximum-deviation branching rules (biased towards the maximum-deviation), and uses incumbent branching whenever possible. 

In our proposed algorithm, we use specialized variable and branching point selection rules, which use information collected from multiple disjunctions and, therefore, take into account the coefficients of the constraints in the model in addition to the variable ranges at the current node.
\paragraph{New proposed rule} Note that we always branch on only one set of variables, either $x$ or $y$. We describe our rule assuming we are branching on the $x$ variables.  To further ease exposition, we explain our proposed branching rules for the root node, i.e., we assume that all variables range from $0$ to $1$. 
Consider the three-variable set:
$$
S_0 = \{(x_{1},y_{1},w_{11})\in{\rr}^3\,|\, q w_{11} + ax_{1} + by_{1} + c = 0, \ w_{11}=x_{1} y_{1}\},
$$
which is obtained by fixing $x_i, y_j$ to either $0$ or $1$ in (\ref{eq_single}), $\forall i\in V_1\setminus\{1\}, \ \forall j\in V_2\setminus \{1\}$. Like the proof of Proposition~\ref{prop: hyp is SOCP in space of w}, there are two cases of interest. 
\begin{itemize}
\item $q\neq 0$. In this case, $w_{11}$ can be written as affine function of $x_1$ and $y_1$. We can then write the projection of $S_0$ in the space of $(x_1,y_1)$ as (we drop the indices to simplify notation, we also drop the word 'Proj')
$$S_0 = \{(x,y)\in[0,1]^2 \,|\, \ (x-r)(y-s) = \tau\},$$
where $r,s,\tau$ are constants. The equation
$(x-r)(y-s) = \tau$ represents a hyperbola with asymptotes $x=r$ and $y=s$. Two typical instances are plotted in Figure~\ref{fig: hyp 2 branches}-\ref{fig: hyp one branch}, where the continuous thick portion of the curves represents $S_0$ and the whole dotted areas represent $\convex(S_0)$.
\begin{figure}[h]
    \centering
    \begin{minipage}{0.48\textwidth}
        \centering
		\includegraphics[scale=0.7]{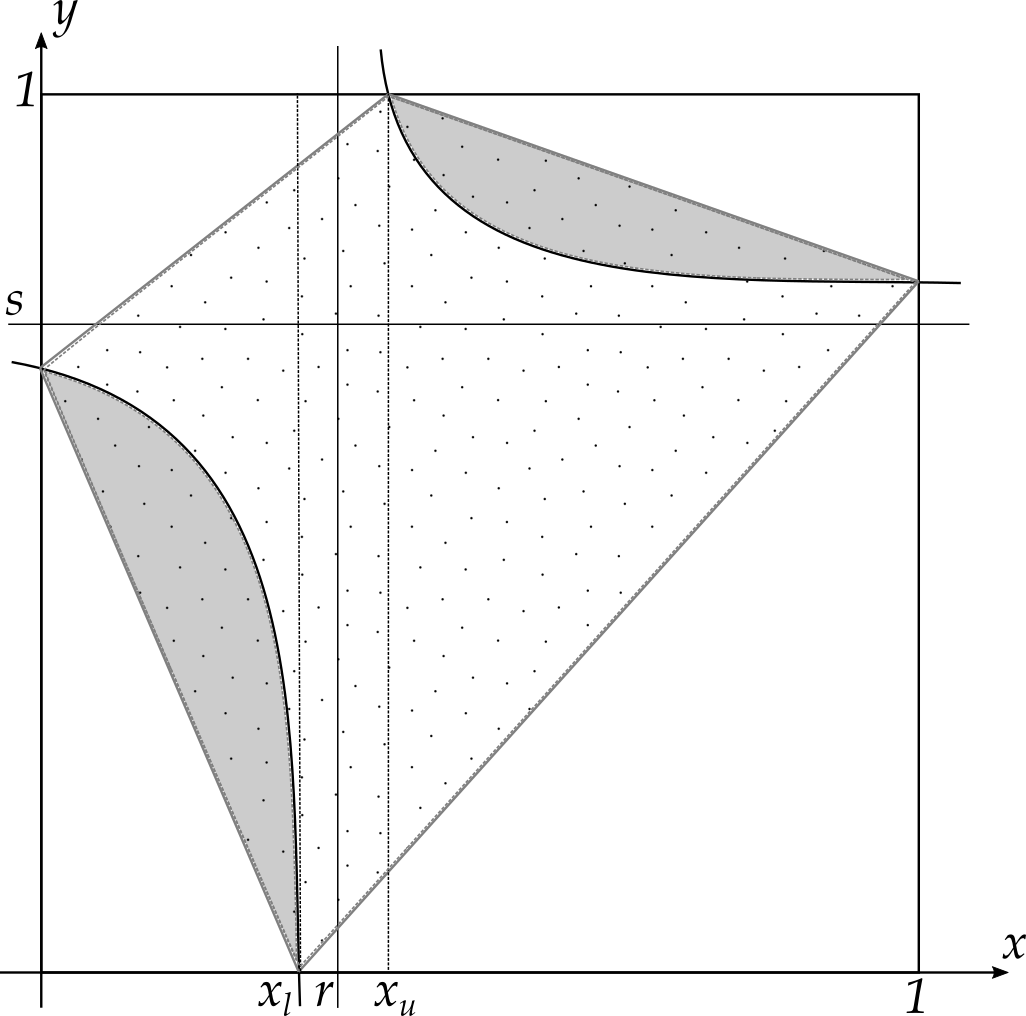}
        \caption{Convex hull of the set defined by the intersection of two branches of a hyperbola with the $[0,1]^2$ box. {Here,} $x_l$ (resp. $x_u$) is the $x$-coordinate of the intersection point of the left (resp. right) branch with the line $y=0$ (resp. $y=1$).\newline}
        \label{fig: hyp 2 branches}
    \end{minipage}%
    \hspace{0.3cm}
    \begin{minipage}{0.48\textwidth}
        \centering
        \includegraphics[scale=0.7]{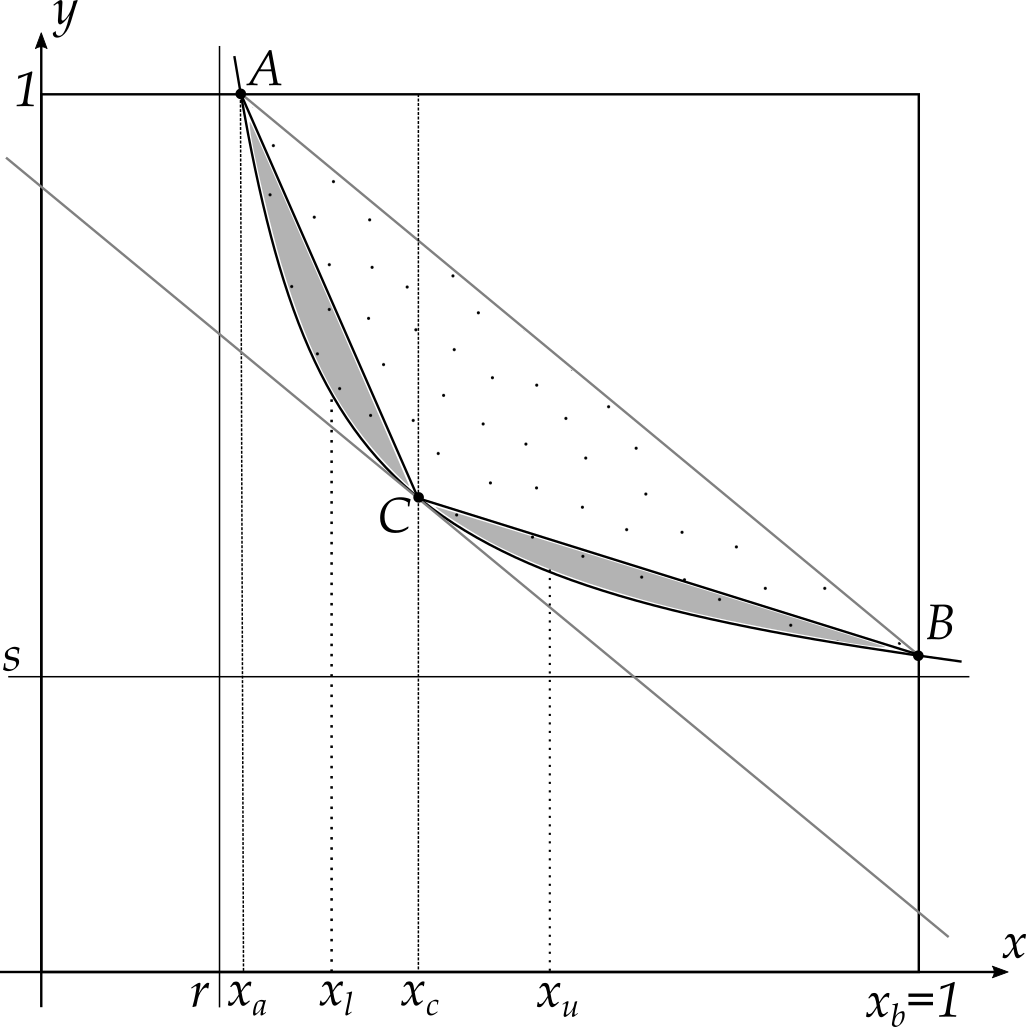} 
        \caption{Convex hull of the set defined by the intersection of a single branch of a hyperbola with the $[0,1]^2$ box.
{Let} $A$ and $B$ are the intersection points of the curve with the $[0,1]^2$ box and $C$ is the point of the curve at which the tangent line is parallel to $AB$. {Then,} $x_a,x_b$ and $x_c$ are the projections of $A,B$ and $C$ onto the $x$ axis.}
        \label{fig: hyp one branch}
    \end{minipage}
\end{figure}
Our goal is to branch at a point that maximizes the eliminated area upon branching. 

Case 1: Both branches of a hyperbola intersect with the $[0,1]^2$ box. Let $x_l$ (resp. $x_u$) be the $x$-coordinate of the intersection point of the left (resp. right) branch with either of the lines $y=0$ or $y=1$. The plot on Figure~\ref{fig: hyp 2 branches} suggests that branching $x$ at any point $x_0\in [x_l,x_u]$ is a reasonable choice for the case where both branches of the hyperbola intersect the $[0,1]^2$ box. Indeed, such branching would eliminate the entire dotted area between the two branches of the curve. 

Case 2: One branch of hyperbola intersects with the $[0, 1]^2$ box. For the case where only one branch intersects the $[0,1]^2$ box, as illustrated in Figure~\ref{fig: hyp one branch}, we could in principle compute $C$ that maximizes the area of the triangle $\bigtriangleup_{ABC}$. To simplify the rule and avoid excessive computations, we simply choose $C$ to be the point at which the tangent line to the curve is parallel to the line $AB$. Moreover, for points in some interval $[x_l,x_u]$ containing $x_c$, the area of the triangle $\bigtriangleup_{ABC}$ does not change much, implying that every point in $[x_l,x_u]$ may be a good choice to branch at. In our computational experiments, we compute $x_l$ and $x_u$ such that $x_c-x_l = \gamma (x_c-x_a)$ and $x_u-x_c = \gamma (x_b-x_c)$ with $\gamma = 2/3$.

\item $q = 0$ and $a\neq 0$ or $b\neq 0$. Without loss of generality assume $b\neq 0$. In this case, $y_{1}$ is an affine function of $x_{1}$ as shown in proof of Proposition~\ref{prop: hyp is SOCP in space of w}. Thus, we can study $S_0$ in the space of $(x_{1},w_{11})$, where it is defined by a parabola and we adopt the same rule defined for the case of Figure~\ref{fig: hyp one branch}, i.e. choose points $x_l$ and $x_u$ as a function of $x_a$ and $x_b$. If the parabola intersects the $[0,1]^2$ box in more than two points, we define $A$ and $B$ to be the left and right most intersection points.

Note that if $a \neq 0$, {then} $x_1$ is an affine function or $y_1$. We can identify appropriate points in the $y_1$ space as above and then translate them to the $x$ space via the affine function.
\end{itemize}

Thus, corresponding to every three-variable set $S_0$, we associate (i) an $x$-variable $x_i$, (ii) an interval $[x_l, x_u]$ within the domain of $x_i$ and (iii) we also approximately compute the area of $\textup{conv}(S_0)$, (either in the space of $(x_1,y_1)$, if $q\neq 0$, or in the space of $(x_1,w_{11})$, if $q=0$), referred to as $A_0$. The actual area we use is that of the polyhedral outer approximation as will be discussed in Section~\ref{sec:Polyhedralrelaxation}.


Once the above data is collected for all disjunctions, we use the following Algorithm to decide on the variable to branch on and the point of partitioning for this variable. 

\begin{algorithm}[H]
\caption{Branching rule}
\label{alg: branching rule}
\begin{algorithmic}[1]
\State \textbf{Input:} $\delta = 1/K$, for some positive integer $K$. Let {$\varepsilon_1,\varepsilon_2 > 0$}. 
\State Let $A_{ik} = 0, \ i \in \{1, \dots, n_1\}, \ k \in \{1, \dots, K\}$. Let $p_i = 0 , \ i \in \{1, \dots, n_1\}$
\State Define $I_{ik} = [(k-1)\delta,k\delta]$, for $k \in \{1, \dots, K\}$ (which defines a partition of the range of $x_i$). 
\For{Each disjunctions $S_0$}
	\State Compute (a) the index $i$ of $x$-variable corresponding to $S_0$, (b) domain $[x_l,x_u]$ and (c) the area $A_0$.
		\State Set $p_i = p_i + 1$ 
		\State If $[x_l,x_u]\cap I_{ik}\neq \emptyset$ for some $k\in\{1, \dots, K\}$, set $A_{ik} = A_{ik} + A_0$.
\EndFor
\State \For{$i \in \{1, \dots, n_1\}$}
\If{$\frac{p_i}{\sum_{l=1}^{n_1}p_l} < {\varepsilon_1}$}
	\State variable $i$ is declared irrelevant.
\EndIf
\EndFor
\State Let $(i^*,k^*) \in \Argmax\{A_{i,k}\,|\, i\in\{1, \dots, n_1\}, i \textup{ is not irrelevant}, k\in \{1,\dots, K\}\}$
\If{$A_{i^*k^*} \geq {\varepsilon_2}$}
	\State Branch on the variable $x_{i^*}$ at the mid point of the interval $I_{i^*k^*}$.
\Else 	
	\State Use the bisection rule.
\EndIf
\end{algorithmic}
\end{algorithm}
In our computational experiments, whenever we use Algorithm~\ref{alg: branching rule}, we set $\varepsilon_1 = 0.01,~\varepsilon_2 = 1/16$ and $K = 8$. Our implementation is naive, and we have not tried to fine tune any of these parameters.

\section{Computational experiments}
\subsection{Finite Element Updating Model}\label{sec: finite element model}
The instances of BBP that we use come from finite element (FE) model updating in structural engineering. The goal is to update the parameter values in an FE model, so that the model provides same resonance frequencies and mode shapes that are physically measured from vibration testing at the as-built structure. In this study we adopt the modal dynamic residual formulation, for which the details can be found in \cite{Wang2015}. The formulation is briefly summarized as follows. 

Consider the model updating of a structure with $m$ number of degrees-of-freedom (DOFs). Corresponding to stiffness parameters that are being updated, the (scaled) updating variables are first denoted as $x\in [-1,1]^{n_1}$. Since only some DOFs can be instrumented, we suppose $n_2$ of those are not instrumented, leaving $m-n_2$ of them as instrumented. In the meantime, it's assumed that $n_3$ number of vibration modes are measured/observed from the vibration testing data. For each $l$-th measured mode, $\forall l\in \{1, \dots, n_3\}$, the experimental results provide $\lambda_l$ as the square of the (angular) resonance frequency, and $\bar{y}^l\in \mathbb{R}^{m-n_2}$ as the mode shape entries at the instrumented DOFs. In mathematical terms, the modal dynamic residual formulation can be stated as the problem of simultaneously solving the following set of equations on stiffness updating variables $x\in [-1,1]^{n_1}$ and (scaled) unmeasured mode shape entries $y^l\in[-2,2]^{n_2}$, $\forall l\in \{1, \dots, n_3\}$:
\begin{align}
[K_0+\sum_{i=1}^{n_1}x_iK_i-\lambda_lM]
	\begin{bmatrix}
	\bar{y}^l\\
	y^l
	\end{bmatrix}
	= 0, \ l\in \{1, \dots, n_3\},\label{setEqs}
\end{align}
where $M,K_0,K_i\in \mathbb{R}^{m\times m}$, $\forall i\in \{1, \dots, n_1\}$,  $\lambda_l\in \mathbb{R}_+$ and $\bar{y}^l\in \mathbb{R}^{m-n_2}, \ \forall l\in \{1, \dots, n_3\}$, are problem data. In practice, (\ref{setEqs}) is unlikely to have a feasible solution set of $x$ and $y^l$, $l\in \{1, \dots, n_3\}$, because of modeling and measurement inaccuracies. Therefore, we convert the problem of solving (\ref{setEqs}) into an optimization problem that aims to minimize the sum of the residuals, i.e., the absolute difference between left and right-hand-side of each equation. After some affine transformations and simplifications, this optimization problem can be stated as following:
\begin{eqnarray}
&\min & \sum_{k=1}^{m} z_{k}\label{P1}\\
&\textup{s.t.}& |x^{\top}Q_{k}y + a_{k}^{\top}x + b_{k}^{\top}y + c_{k}| = z_{k}, \ k\in \{1, \dots, m\}\nonumber\\
& & x\in[0,1]^{n_1},y\in[0,1]^{n_2},\nonumber
\end{eqnarray}
where $n_2$ and $m$ correspond to {$n_2n_3$} and $mn_3$, respectively, in the notation of (\ref{setEqs}). Finally, (\ref{P1}) is equivalent to the following BBP.
\begin{eqnarray}
&\min & \sum_{k=1}^{m} z'_{k}+ z''_{k}\label{P2}\\
&\textup{s.t.}& x^{\top}Q_{k}y + a_{k}^{\top}x + b_{k}^{\top}y + c_{k} = z'_{k}-z''_{k}, \ k\in \{1, \dots, m\}\nonumber\\
& & x\in[0,1]^{n_1},y\in[0,1]^{n_2}.\nonumber\\
& & 0 \leq z'_{k}, z''_{k} \leq u, \ k \in \{1, \dots, m\}.\nonumber
\end{eqnarray}

\noindent\textbf{Instances:}\\
The simulated structural example is similar to the planar truss structure in \cite{Wang2015}. In order to simulate measurement noise, we add a normal-distributed random variable to the parameters $\lambda^l$ and $\bar{y}^l$, ${\forall} l\in\{1,\dots,n_3\}$, with mean zero and variance equal $2\%$ of its actual value. {In our case there are six modes, i.e. $n_3=6$}. By taking different values for {$n_2$}, we then generate ten instances whose number of variables and constraints are given in Table~\ref{table: inst}.
\begin{table}[h]
\centering
\caption{Instances description}
\label{table: inst}
\begin{tabular}{ccccc}
Inst & $\#$ of x-variables & $\#$ of y-variables & $\#$ of equations &  $\#$ of bilinear terms\\
\hline
inst1    & 6                   & 180                 & 312    &	990           \\
inst2    & 6                   & 180                 & 312    &	954           \\
inst3    & 6                   & 168                 & 312    &	966           \\
inst4    & 6                   & 168                 & 312    &	972           \\
inst5    & 6                   & 156                 & 312    &	900           \\
inst6    & 6                   & 144                 & 312    &	780          
\\
inst7    & 6                   & 132                 & 312    &	756          
\\
inst8    & 6                   & 132                 & 312    &	756          
\\
inst9    & 6                   & 120                 & 312    &	684          
\\
inst10    & 6                   & 120                 & 312    &	684          
\end{tabular}
\end{table}

\subsection{Simplifying $S^{SOCP}$}
\subsubsection{A lighter version of $S^{SOCP}$}\label{sec: lighter version}
According to Remark~\ref{rem:size}, the number of disjunction needed to model the convex hull of a single bilinear equation can be computationally prohibitive for many instances of interest. To overcome this issue, in our computational experiments, we write the convex hull of each row only in the space of the variable appearing in it. In particular, for constraint $k$ we work with $G(V^k, E^k)$, where $V^k$ is the set of variables appearing in constraint $k$ and $E^k$ represent the complete bipartite graph between the $x$ and $y$ variables appearing in $V^k$. This possibly weaker relaxation is much more computationally cheaper that $S^{SOCP}$ for our instances due to the sparsity on the coefficients of each bilinear equation. We denote this relaxation as $\textup{light}-S^{SOCP}$.

\subsubsection{Polyhedral outer approximation}\label{sec:Polyhedralrelaxation}

As shown in Proposition \ref{prop: hyp is SOCP} and Proposition \ref{prop: parabola is SOCP}, all the sets obtained after fixings are SOCP representable. Some are polyhedral while many of the others are not. Since linear programming techniques are more efficient and robust, than the non-linear counterpart, we outer approximate the non-polyhedral sets by polyhedral sets.

As shown in proof of Proposition~\ref{thm: convex of single eq is SOCP in the space of w}, all the non-linear sets that we need to convexify in order to obtain the convex hull of the set ${S}$ defined in (\ref{eq_single}) are of the form
\begin{align*}
S_{i_0j_0} = \{(x,y,w)\in[0,1]^{n_1+n_2+n_1n_2}\,|\, x_i, y_j\in\{0,1\}, \ \forall i\in V_1\setminus\{i_0\}, \ \forall j\in V_2\setminus \{j_0\}, \\ \bar{q} w_{i_0j_0} + \bar{a}x_{i_0} + \bar{b}y_{i_0} + \bar{c} = 0 , \ w_{ij} = x_{i} y_{j}, \ i \in V_1, \ j\in V_2\},
\end{align*}
for some $(i_0,j_0)\in E$.
Without loss of generality, suppose $i_0=1$ and $j_0=1$, in which case we want to outer approximate the following set
$
S_0 = \{(x_{1},y_{1},w_{11})\in{\rr}^3\,|\, q w_{1} + ax_{1} + by_{1} + c = 0, \ w_{11}=x_{1} y_{1}\}.
$
There are two cases of interest.  
The first case occurs when $q\neq 0$. In this case, $w_{11}$ is an affine functions of $x_{1}$ and $y_{1}$ as following: $w_{11} = (-c-ax_{1}-by_{1})/q$; $w_{1j} = x_{1}y_j, \ \forall j\in \{1, \dots, n_2\}$; $w_{i1} = x_{i}y_{1}, \ \forall i\in \{1, \dots, n_1\}$; and $w_{ij}=x_iy_j, \ \forall i\in \{1, \dots, n_1\}\setminus\{1\}, \ \forall j\in \{1, \dots, n_2\}\setminus\{1\}$. Hence, we only need to approximate $\convex(S_0)$ in the space of $(x_{1},y_{1})$. If both branches of the hyperbola defined by $q x_{1} y_{1} + ax_{1} + by_{1} + c = 0$ intersect the $[0,1]^2$ box, than $\convex(S_0)$ is polyhedral. Suppose only one branch of the hyperbola intersects the box. Then, we outer approximate $\convex(S_0)$ by using tangent lines to the curve. In our implementation, we only use the tangent lines at the intersection points of the curve with the box, see Figure~\ref{fig: hyperbola}. More tangent lines could be added to better approximate $\convex(S_0)$, but based on our preliminary experience on our instances it does not make significant difference. 

The second case of interest is $q = 0$ and $a\neq 0$ (or $b\neq 0$) for which we can rewrite $S_0$ as
$
S_0 = \{(x_{1},y_{1},w_{11})\in[0,1]^3\,|\, aw_{11} = -by_{1}^2-cy_{1}, \ ax_{1} = -by_{1}-c\}.
$
In this case, $x_{1}$ is an affine function of $y_{1}$ and we only need to approximate $\convex(S_0)$ in the space of $(y_{1},w_{11})$, where $ aw_{11} = -cy_{1}-by_{1}^2$ defines a parabola as shown in Figure~\ref{fig: parabola}. As in the previous case, we outer approximate the curve by using tangent lines to the curve as illustrated in Figure~\ref{fig: parabola}.
\begin{figure}[!h]
    \centering
    \begin{minipage}{0.48\textwidth}
        \centering
	\includegraphics[scale=0.8]{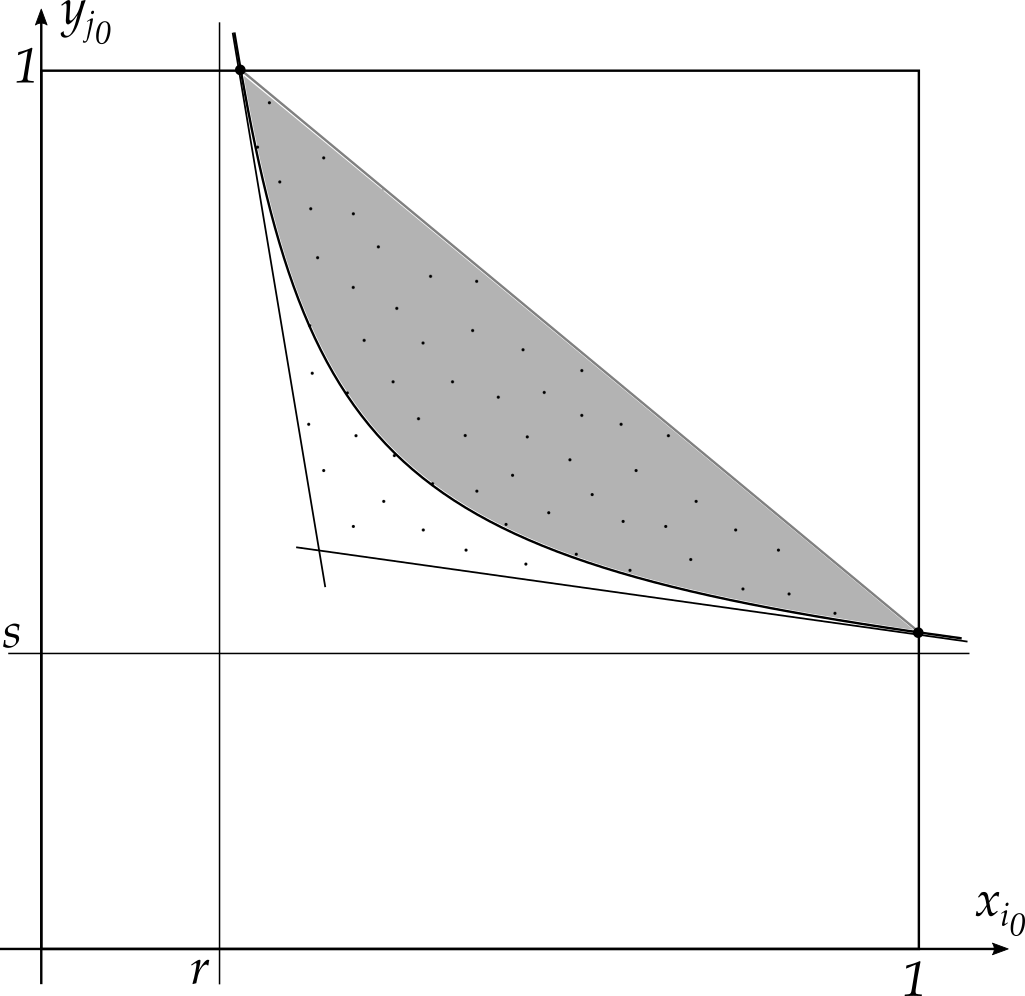}
        \caption{Convex hull of the set defined by the intersection of one branch of a hyperbola with the $[0,1]^2$ box, and its tangential linear outer approximation.}
        \label{fig: hyperbola}
    \end{minipage}%
    \hspace{0.3cm}
    \begin{minipage}{0.48\textwidth}
        \centering
        \includegraphics[scale=0.8]{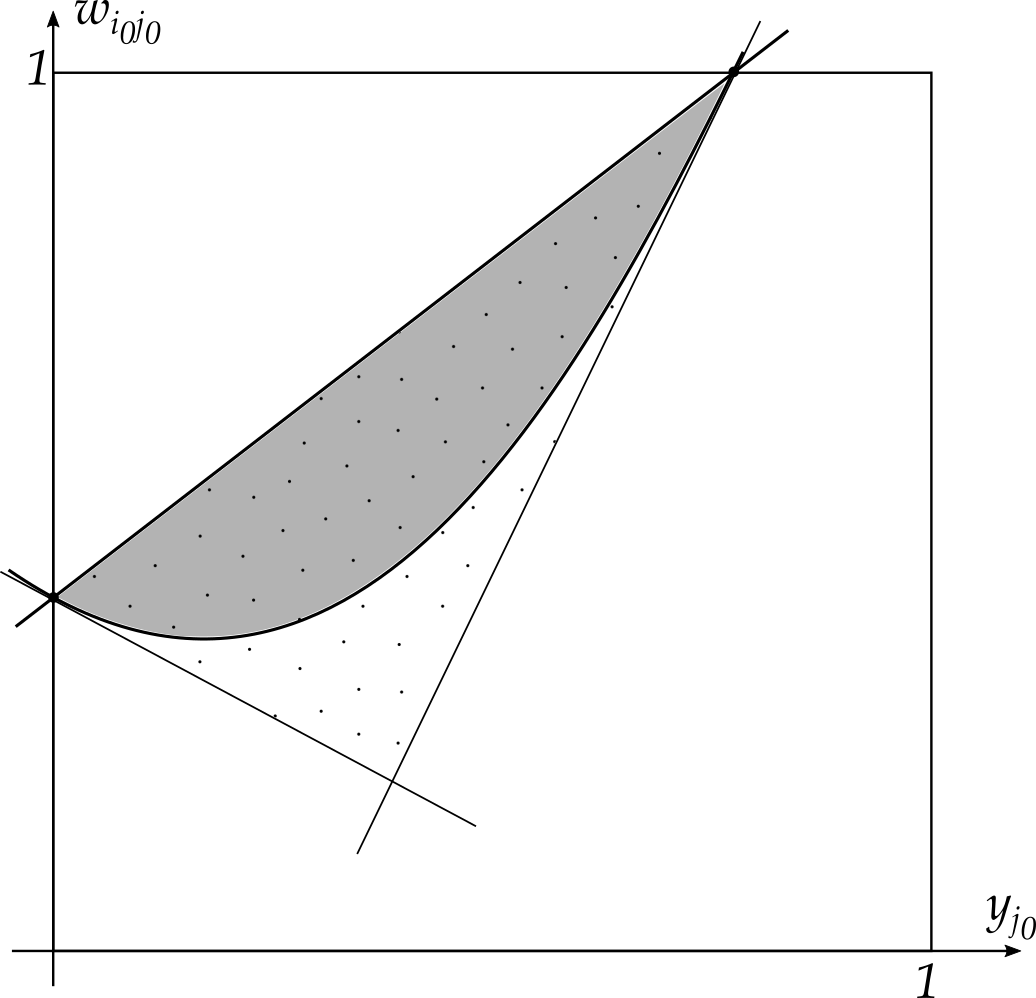} 
        \caption{Convex hull of the set defined by the intersection of parabola with the $[0,1]^2$ box, and its tangential linear outer approximation.\newline}
        \label{fig: parabola}
    \end{minipage}
\end{figure}

\subsection{Computation results}\label{sec:compres}

\subsubsection{Software and Hardware}
All of our experiments were ran on a Windows 10 machine with 64-bit operating system, x64 based processor with 2.19GHz, and 32GB RAM. We call MOSEK via CVX from MATLAB R2015b to solve SDPs. We used Gurobi 7.5.1 to solve LPs and integer programs. We used BARON 15.6.5 (with CPLEX 12.6 as LP solver and IPOPT as nonlinear solver) as our choice of commercial global solver, which we call from MATLAB R2015b.
\subsubsection{Root node}
We assess the strength of our proposed polyhedral outer approximation of $\textup{light}-S^{SOCP}$ relaxation (defined in Section~\ref{sec: lighter version} and referred as SOCP in the tables) against the classical SDP and McCormick (Mc) relaxations. The numerical results are reported in Table~\ref{table:Root relaxation comparisons}, where SDP+Mc denotes the the intersection of SDP and Mc relaxations. Similarly, SOCP+Mc denotes the intersection of SOCP and Mc relaxations (since we are not using $S^{SOCP}$, this could potentially be stronger than SOCP). 
\begin{table}[h]
\centering
\caption{Root relaxations}
\label{table:Root relaxation comparisons}
\begin{tabular}{ccccccccccc}
     & \multicolumn{2}{c}{Mc} & \multicolumn{2}{c}{SDP} & \multicolumn{2}{c}{SDP+Mc} & \multicolumn{2}{c}{SOCP} & \multicolumn{2}{c}{SOCP+Mc} \\
Inst & Bound       & Time     & Bound        & Time     & Bound         & Time       & Bound        & Time      & Bound         & Time        \\
     \hline
1    & 0.17771     & 0.07     & 0.17771      & 1.81     & 0.17771       & 35.89      & 0.17793      & 17.59     & 0.17793       & 18.42       \\
2    & 0.00000     & 0.05     & 0.00000      & 1.70     & 0.00000       & 38.98      & 0.00000      & 20.93     & 0.00000       & 21.14       \\
3    & 0.27543     & 0.07     & 0.27194      & 1.81     & 0.27543       & 44.02      & 0.28202      & 16.22     & 0.28202       & 49.61       \\
4    & 0.10095     & 0.08     & 0.10012      & 2.14     & 0.10095       & 36.13      & 0.10101      & 20.71     & 0.10101       & 25.87       \\
5    & 0.34766     & 0.05     & 0.34766      & 1.67     & 0.34766       & 31.58      & 0.34925      & 13.17     & 0.34925       & 12.88       \\
6    & 0.97758     & 0.05     & 0.91629      & 1.80     & 0.97758       & 28.47      & 1.00267      & 11.64     & 1.00267       & 11.07      
\\
7  & 1.73437 & 0.07 & 1.70329 & 1.38 & 1.73437 & 25.29 & 1.74015 & 10.76 & 1.74015 & 11.68 \\
8  & 1.99887 & 0.07 & 1.97107 & 1.30 & 1.99887 & 21.95 & 2.01260 & 17.53 & 2.01260 & 21.51 \\
9  & 1.89400 & 0.05 & 1.89222 & 1.17 & 1.89400 & 22.94 & 1.90191 & 10.53 & 1.90191 & 9.32  \\
10 & 2.41036 & 0.05 & 2.40658 & 1.16 & 2.41036 & 18.95 & 2.41959 & 10.07 & 2.41959 & 12.29
\end{tabular}
\end{table}

As we see, SOCP produces the best dual bounds among SDP, Mc and SDP + Mc. Also, SOCPs runs faster than SDP + Mc for all the instances. Finally, SOCP+ Mc produces no better bounds than SOCPs alone. 

A strong relaxation can be obtained by partitioning the domain of some variables and writing a MILP formulation to model the union of McCormick relaxations over each piece~\cite{meyer2006global, dey2015analysis}.  We call it McCormick Discretization and use the MILP formulation with binary expansion. We only partition the domain of variables $x_i$'s as the number of $x$ variables is much smaller than the number of $y$ variables for all of our instances. In Table~\ref{table: Mc Disc}, $T$ defines the level of discretization, meaning that the range of each variable $x_i$ is partitioned into $2^T+1$ uniform sub-intervals. This relaxation becomes tighter as $T$ increases. However, the MILP that need to be solved becomes harder since the number of binary variables increases as a function of $T$. Thus, we give GUROBI a time limit of 10 hours, which is the amount of time given to all the branch-and-bound algorithm that we report in Section~\ref{sec: branch and  bound} below. Table~\ref{table: Mc Disc} reports the computational results, where the asterisk signalizes that GUROBI reached the time limit with the given level of discretization. {If this is the case, then} we report the MILP dual bound reported by the solver, which is a valid dual bound for our problem. The last column displays the best bound obtained among all the levels of discretizations reported.

\begin{table}[h]
\centering
\caption{McCormick discretization: dual bounds}
\label{table: Mc Disc}
\begin{tabular}{cccccccc}
Inst & T=6     & T=8     & T=10    & T=12*   & T=14*   & T=16*   & Best    \\
\hline
1    & 0.18611 & 0.20512 & 1.11852 & 1.85387 & 1.40586 & 0.96121 & 1.85387 \\
2    & 0.00000 & 0.03133 & 1.05662 & 2.14709 & 1.38374 & 0.04654 & 2.14709 \\
3    & 0.29443 & 0.33575 & 1.39375 & 2.14270 & 1.42642 & 1.42007 & 2.14270 \\
4    & 0.10524 & 0.11387 & 1.21446 & 2.44853 & 1.63495 & 1.27218 & 2.44853 \\
5    & 0.36159 & 0.47559 & 2.15416 & 3.40272 & 3.22915 & 2.67721 & 3.40272 \\
6    & 1.25052 & 2.61325 & 4.16459 & 4.06782 & 3.96512 & 3.78165 & 4.16459 \\
7    & 1.96682 & 2.17988 & 3.60737 & 4.92133 & 4.69632 & 4.47471 & 4.92133 \\
8    & 2.48886 & 2.69510 & 3.63400 & 4.81890 & 4.48014 & 4.19095 & 4.81890 \\
9    & 2.05584 & 2.42150 & 4.16064 & 5.54076 & 5.63110 & 5.15290 & 5.63110 \\
10   & 2.57751 & 2.80795 & 4.07475 & 5.40977 & 5.28173 & 5.16376 & 5.40977
\end{tabular}
\end{table}
Clearly, McCormick discretization produces better results than $SOCP$. Therefore, if one does not want to use branch and bound, then McCormick discretization is the best option. However, as we see in the next section, better dual bounds can be obtained by combining SOCP with the new proposed branch-and-bound algorithm.

\subsubsection{Branch-and-bound}\label{sec: branch and  bound}

We assess and compare the performance of the following methods:\\
- \textit{BB}: This stands for our implementation of a branch-and-bound algorithm coded in Python. We use GUROBI as LP solver and run IPOPT at each node to search for feasible solutions. Our algorithm uses best-bound-first as node selection and rectangular partitioning. We consider three variants that differ from each other based on the relaxation adopted in each node and in the way variables and branching points are selected:
\begin{itemize}
\item[-] \textit{SOCP-1}: Uses the polyhedral relaxation described in Section~\ref{sec:Polyhedralrelaxation} with variable selection and the branching point given by Algorithm~\ref{alg: branching rule}.
\item[-] \textit{SOCP-2}: Uses the same relaxation of BB-SOCP-1 above. The branching variable is selected according to the gap-error-rule explained in Section~\ref{sec: variable and branching point selection}. Then uses the incumbent-rule for branching point selection, whenever possible, otherwise uses the maximum-deviation-rule.
\item[-] \textit{SOCP-3}: Same as BB-SOCP-2 except that uses bisection for branching point selection.
\item[-] \textit{BB-Mc}: Uses McCormick relaxation with gap-error-rule as branching variable selection rule and bisection for branching point selection.
\end{itemize}
The dual bounds from our computational experiments are reported in Table~\ref{table: dual bounds}. The stopping criteria for all the methods was a time limit of 10 hours.

\begin{table}[h]
\centering
\caption{Branch-and-bound methods: dual bounds}
\label{table: dual bounds}
\begin{tabular}{ccccc}
Inst & BB-SOCP-1 & BB-SOCP-2 & BB-SOCP-3 & BB-Mc      \\
\hline
1    & 2.50744   & 0.18473   & 0.18228   & 0.18343 \\
2    & 2.86438   & 0.00000   & 0.00000   & 0.00000 \\
3    & 3.13078   & 0.29109   & 0.28983   & 0.28884 \\
4    & 3.11154   & 0.10526   & 0.10246   & 0.10410 \\
5    & 3.78958   & 0.35253   & 0.35392   & 0.35405 \\
6    & 4.63992   & 1.11105   & 1.09537   & 1.15191 \\
7  & 5.26603 & 1.99569 & 1.88331 & 1.94949 \\
8  & 5.13128 & 2.18546 & 2.18193 & 2.28761 \\
9  & 6.10860 & 2.17509 & 2.08068 & 2.10144 \\
10 & 5.77051 & 2.48039 & 2.45158 & 2.47965
\end{tabular}
\end{table}

The best dual bound for each instance is clearly given by BB-SOCP-1, which uses our proposed relaxation and branching rule. All the standard branching rules yield significantly worse bounds.

\subsubsection{McCormick relaxation with BB-SOCP-1 branching rules}

The computational results from Section~\ref{sec: branch and  bound}, suggest that the good performance of BB-SOCP-1 is highly dependent on its branching rules, defined according to Algorithm~\ref{alg: branching rule}. In this section we show that the branching rules of Algorithm~\ref{alg: branching rule} on them own are not enough to produce good dual bounds. 

Consider the variant of BB-SOCP-1, reffered as BB-SOCP-Mc, which uses only McCormick relaxation and the same branching rule given by Algorithm~\ref{alg: branching rule}. Thus, at each node, we collect data from each disjunction $S_0$, run Algorithm~\ref{alg: branching rule} to select the branching variable and the branching point, but we only use the McCormick inequalities to define the relaxation. 

In Table~\ref{table: SOCP1 vc SOCP_Mc}, we compare the performance of BB-SOCP-1 and BB-SOCP-Mc. It becomes clear that the strength of BB-SOCP-1 does not come only from the branching rules of Algorithm~\ref{alg: branching rule} but also from our proposed relaxation. The discrepancy in the performance of BB-SOCP-1 and BB-SOCP-Mc means that, as the algorithm goes down the tree, the SOCP relaxation becomes much tighter than the McCormick relaxation.

\begin{table}[h]
\centering
\caption{BB-SOCP-1 vs. McCormick relaxation with BB-SOCP-1 branching rules}
\label{table: SOCP1 vc SOCP_Mc}
\begin{tabular}{ccccc}
     & \multicolumn{2}{c}{BB-SOCP-1} & \multicolumn{2}{c}{BB-SOCP-Mc} \\
Inst & Dual Bound & Gap ($\%$) & Dual Bound & Gap ($\%$) \\
\hline
1    & 2.50744    & 27.9       & 0.19776    & 94.3       \\
2    & 2.86438    & 18.2       & 0.02752    & 99.2       \\
3    & 3.13078    & 14.9       & 0.30514    & 91.7       \\
4    & 3.11154    & 17.1       & 0.11188    & 97.0       \\
5    & 3.78958    & 8.3        & 0.40497    & 90.2       \\
6    & 4.63992    & 18.0       & 1.52070    & 73.1       \\
7    & 5.26603    & 6.0        & 2.26765    & 59.5       \\
8    & 5.13128    & 9.5        & 2.68861    & 52.6       \\
9    & 6.10860    & 1.5        & 2.51461           & 59.5       \\
10   & 5.77051    & 7.9        & 2.85232    & 54.2      
\end{tabular}
\end{table}

\subsubsection{Comparison of primal bounds and duality gaps}

Finally, we report in Table~\ref{table: primal bounds} a summary of the performance of BB-SOCP-1, McCormick Discretization, BARON and BB-Mc. Recall that the stopping criteria for all the methods was a time limit of 10 hours. Also recall that primal solutions for BB-SOCP-1 and BB-Mc are obtained using IPOPT.


\begin{table}[h]
\centering
\caption{Primal bounds and duality gaps}
\label{table: primal bounds}
\resizebox{\textwidth}{!}{
\begin{tabular}{cccccccccccc}
 & \multicolumn{3}{c}{BB-SOCP-1}  & \multicolumn{2}{c}{Mc Disc} & \multicolumn{3}{c}{BARON}      & \multicolumn{3}{c}{BB-Mc}      \\
Inst & Dual    & Primal  & Gap($\%$) & Dual        & Gap($\%$)    & Dual    & Primal  & Gap($\%$) & Dual    & Primal  & Gap($\%$) \\
\hline
1    & 2.50744 & 3.47847 & 27.9       & 1.85387     & 46.7          & 0.33122 & 3.47887 & 90.5       & 0.18343 & 3.47849 & 94.7       \\
2    & 2.86438 & 3.49983 & 18.2       & 2.14709     & 38.6          & 0.52447 & 3.49931 & 85.0       & 0.00000 & 3.49983 & 100.0      \\
3    & 3.13078 & 3.68103 & 14.9       & 2.14270     & 41.8          & 0.47599 & 3.68306 & 87.1       & 0.28884 & 3.73308 & 92.3       \\
4    & 3.11154 & 3.75223 & 17.1       & 2.44853     & 34.7          & 0.78630 & 3.75297 & 79.0       & 0.10410 & 3.75225 & 97.2       \\
5    & 3.78958 & 4.13277 & 8.3        & 3.40272     & 17.7          & 0.38396 & 4.13541 & 90.7       & 0.35405 & 4.28165 & 91.7       \\
6    & 4.63992 & 5.66096 & 18.0       & 4.16459     & 26.4          & 2.26566 & 5.66053 & 60.0       & 1.15191 & 5.66096 & 79.7      \\
7  & 5.26603 & 5.60009 & 6.0  & 4.92133 & 12.1 & 3.07096 & 5.60020 & 45.2 & 1.94949 & 5.69318 & 65.8 \\
8  & 5.13128 & 5.67022 & 9.5  & 4.81890 & 15.0 & 2.70237 & 5.67025 & 52.3 & 2.28761 & 5.67252 & 59.7 \\
9  & 6.10860 & 6.20343 & 1.5  & 5.63110 & 9.2  & 3.67301 & 6.20346 & 40.8 & 2.10144 & 6.29365 & 66.6 \\
10 & 5.77051 & 6.26853 & 7.9  & 5.40977 & 13.1 & 2.94060 & 6.22639 & 52.8 & 2.47965 & 6.30477 & 60.7
\end{tabular}}
\end{table}

The primal bounds from all the three branch-and-bound methods are similar, suggesting that the solutions found are close to a global optimal. On the other hand, the dual bounds from BB-SOCP-1 are significantly better than the dual bounds from all the other methods, which can be seem by comparing the duality gaps. In particular, the duality gap from BB-SOCP-1 is considerably smaller than the duality gap from Mc Disc, even though we are reporting the best dual bound obtained among all the levels of discretizations $T=6,8,\cdots,16$, and the primal bound we use to compute the duality gap of Mc Disc is the best primal bound from BB-SOCP-1, BARON and BB-Mc. The standard branching, i.e., the McCormick relaxation with bisection, yields the worse performance for all the instances.

\section*{Acknowledgments}

The authors would like to thank Xinjun Dong in Civil and Environmental Engineering at Georgia Tech, for his
assistance with preparing the structural example data. Santanu S. Dey would like to acknowledge the discussion on a preliminary version of this paper at Dagstuhl workshop \# 18081, that helped improve the paper. 

Funding: This work was supported by the NSF CMMI [grant number 1149400]; the NSF CMMI [grant number 1150700]; and the CNPq-Brazil [grant number 248941/2013-5].

\bibliographystyle{plain}
\bibliography{sensor_model_bibfile_8}

\end{document}